\documentclass{article}
\usepackage{amsfonts, amstext, amsmath, amsthm, amscd, amssymb, hyperref, float, longtable, enumitem, fullpage}
\usepackage{epsfig, graphics, psfrag}
\usepackage{graphicx}
\usepackage{color}
\usepackage[font=small, labelfont=bf]{caption}
\usepackage{float}

\newtheorem{definition}{Definition}[subsection]
\let\olddefinition\definition
\renewcommand{\definition}{\olddefinition\normalfont}
\newtheorem{remark}{Remark}[subsection]
\let\oldremark\remark
\renewcommand{\remark}{\oldremark\normalfont}
\newtheorem{proposition}{Proposition}[subsection]
\newtheorem{theorem}{Theorem}[subsection]
\newtheorem{lemma}{Lemma}[subsection]
\newtheorem{corollary}{Corollary}[subsection]

\begin{document}
\title{Combinatorics of Link Diagrams and Volume}
\author{Adam Giambrone\thanks{Research supported in part by RTG grant DMS-0739208 and NSF grant DMS-1105843.}\\ Alma College\\ giambroneaj@alma.edu}
\date{}
\maketitle
\begin{abstract}
We show that the volumes of certain hyperbolic A-adequate links can be bounded (above and) below in terms of two diagrammatic quantities: the twist number and the number of certain alternating tangles in an A-adequate diagram. We then restrict our attention to plat closures of certain braids, a rich family of links whose volumes can be bounded in terms of the twist number alone. Furthermore, in the absence of special tangles, our volume bounds can be expressed in terms of a single stable coefficient of the colored Jones polynomial. Consequently, we are able to provide a new collection of links that satisfy a Coarse Volume Conjecture.
\end{abstract}

\section{Introduction}

One of the current aims of knot theory is to strengthen the relationships among the hyperbolic volume of the link complement, the colored Jones polynomials, and data extracted from link diagrams. Recently, Futer, Kalfagianni, and Purcell (\cite{Survey}, \cite{Guts}) showed that, for sufficiently twisted negative braid closures and for certain Montesinos links, the volume of the link complement can be bounded above and below in terms of the twist number of an A-adequate link diagram. Similar results for alternating links were found in \cite{Lackenby} and improved upon in the appendix of \cite{Lackenby} and in \cite{AgolStorm}. The volume of many families of link complements has also been expressed in terms of coefficients of the colored Jones polynomial (\cite{Volumish}, \cite{Guts}, \cite{Filling}, \cite{Symmetric}, \cite{Cusp}, \cite{Stoimenow}).

In this paper, we begin with a study of the structure of A-adequate link diagrams whose all-A states satisfy a certain two-edge loop condition. We use this study to express a lower bound on the volume of the link complement in terms of two diagrammatic quantities: the twist number and the number of certain alternating tangles (called \emph{special tangles}) in the A-adequate diagram. This result complements the work of Agol and D. Thurston (\cite{Lackenby}, Appendix), in which the volume is bounded above in terms of the twist number alone. It should also be noted that the recent work of Futer, Kalfagianni, and Purcell in \cite{New} shows that the links considered in this paper must be hyperbolic. Let $t(D)$ denote the twist number of $D(K)$ and let $st(D)$ denote the number of special tangles in $D(K)$. The main result of this paper is stated below: 

\begin{theorem}[Main Theorem] Let $D(K)$ be a connected, prime, A-adequate link diagram that satisfies the two-edge loop condition and contains $t(D)\geq2$ twist regions. Then $K$ is hyperbolic and the complement of $K$ satisfies the following volume bounds:

\begin{equation}
\frac{v_{8}}{3}\cdot\left[t(D)-st(D)\right] \leq \mathrm{vol}(S^{3}\backslash K) < 10v_{3}\cdot(t(D)-1),
\end{equation}

\noindent where $t(D)\geq st(D)$. If $t(D)=st(D)$, then $D(K)$ is alternating and the lower bound of $\displaystyle \frac{v_{8}}{2}\cdot (t(D)-2)$ from Theorem $2.2$ of \cite{AgolStorm} may be used. Recall that $v_{8}=3.6638\ldots$ and $v_{3}=1.0149\ldots$ denote the volumes of a regular ideal octahedron and tetrahedron, respectively. 

\label{mainthm} 
\end{theorem}

Note that the coefficients of $t(D)$ in the upper and lower bounds differ by a multiplicative factor of $8.3102\ldots$, a factor that we would like to reduce by studying specific families of links. Therefore, we will later restrict attention to A-adequate plat closures of certain braids (which we call \emph{strongly negative plat diagrams} and \emph{mixed-sign plat diagrams}). By studying the structure of these two families of link diagrams, we can provide volume bounds that are usually sharper than those given by the Main Theorem. 

Furthermore, we are able to translate the volume bounds of the Main Theorem so that they may be expressed in terms of $st(D)$ and a single stable coefficient, $\beta_{K}'$, of the colored Jones polynomial. In many cases, the volume of the strongly negative and mixed-sign plats can be bounded in terms of $\beta_{K}'$ alone. Results of this nature can be viewed as providing families of links that satisfy a Coarse Volume Conjecture (\cite{Guts}, Section 10.4).


\section{Preliminaries}
\label{sec}

Let $D(K) \subseteq S^2$ denote a diagram of a link $K \subseteq S^3$. To smooth a crossing of the link diagram $D(K)$, we may either \emph{A-resolve} or \emph{B-resolve} this crossing according to Fig.~\ref{resolutions}. By A-resolving each crossing of $D(K)$ we form the \emph{all-A state} of $D(K)$, which is denoted by $H_{A}$ and consists of a disjoint collection of \emph{all-A circles} and a disjoint collection of dotted line segments, called \emph{A-segments}, that are used record the locations of crossing resolutions. We will adopt the convention throughout this paper that any unlabeled segments are assumed to be A-segments. We call a link diagram $D(K)$ \emph{A-adequate} if $H_{A}$ does not contain any A-segments that join an all-A circle to itself, and we call a link $K$ \emph{A-adequate} if it has a diagram that is A-adequate. 

\begin{figure}
	\centering
		\def\svgwidth{2.5in}
		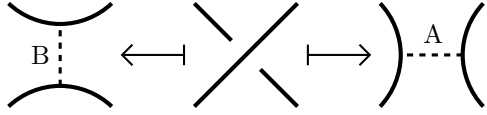
	\caption{A crossing neighborhood of a link diagram (middle), along with its A-resolution (right) and B-resolution (left).}
	\label{resolutions}
\end{figure}

\begin{remark}
While we will focus exclusively on A-adequate links, our results can easily be extended to semi-adequate links by reflecting the link diagram $D(K)$ and obtaining the corresponding results for B-adequate links. 
\end{remark}

From $H_{A}$ we may form the \emph{all-A graph}, denoted $\mathbb{G}_{A}$, by contracting the all-A circles to vertices and reinterpreting the A-segments as edges. From this graph we can form the \emph{reduced all-A graph}, denoted $\mathbb{G}_{A}'$, by replacing all multi-edges with a single edge. For an example of a diagram $D(K)$, its all-A resolution $H_{A}$, its all-A graph $\mathbb{G}_{A}$, and its reduced all-A graph $\mathbb{G}_{A}'$, see Fig.~\ref{figure8}. Let $v(G)$ and $e(G)$ denote the number of vertices and edges, respectively, in a graph $G$. Let $-\chi(G)=e(G)-v(G)$ denote the negative Euler characteristic of $G$. 

\begin{figure}
	\centering
		\def\svgwidth{3.5in}
		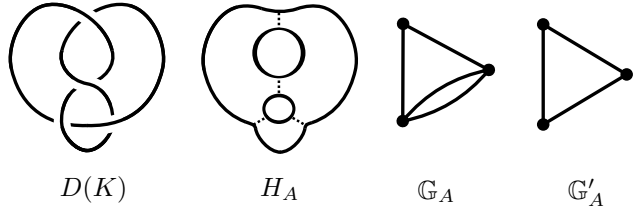
	\caption{A link diagram $D(K)$, its all-A resolution $H_{A}$, its all-A graph $\mathbb{G}_{A}$, and its reduced all-A graph $\mathbb{G}_{A}'$.}
	\label{figure8}
\end{figure}

\begin{remark}
\label{circleremark}
Note that $v(\mathbb{G}_{A}')$ is the same as the number of all-A circles in $H_{A}$ and that $e(\mathbb{G}_{A})$ is the same as the number of A-segments in $H_{A}$. From a graphical perspective, A-adequacy of $D(K)$ can equivalently be defined by the condition that $\mathbb{G}_{A}$ contains no one-edge loops that connect a vertex to itself.   
\end{remark}

\begin{figure}
	\centering
		\def\svgwidth{200pt}
		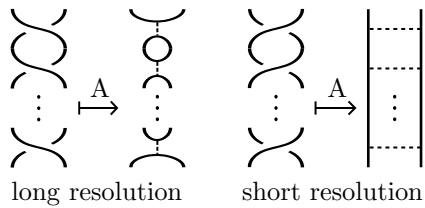
	\caption{Long and short resolutions of a twist region of $D(K)$.}
	\label{longshort}
\end{figure}

\begin{definition} 
Define a \emph{twist region} of $D(K)$ to be a longest possible string of bigons in the projection graph of $D(K)$. Denote the number of twist regions in $D(K)$ by $t(D)$ and call $t(D)$ the \emph{twist number} of $D(K)$. Note that it is possible for a twist region to consist of a single crossing of $D(K)$. 
\end{definition} 

\begin{definition} 
If a given twist region contains two or more crossings, then the A-resolution of a left-handed twist region will be called a \emph{long resolution} and the A-resolution of a right-handed twist region will be called a \emph{short resolution}. See Fig.~\ref{longshort} for depictions of these resolutions. We will call a twist region \emph{long} if its A-resolution is long and \emph{short} if its A-resolution is short. 
\end{definition}

\begin{definition}
A link diagram $D(K)$ satisfies the \emph{two-edge loop condition (TELC)} if, whenever two all-A circles share a pair of A-segments, these segments correspond to crossings from the same short twist region of $D(K)$. 
\end{definition}

\begin{definition}
Call an alternating tangle in $D(K)$ a \emph{special tangle} if, up to planar isotopy, it consists of exactly one of the following:
\begin{itemize}


\item[(1)] a tangle sum of a vertical short twist region and a one-crossing twist region (with the crossing type of Fig.~\ref{resolutions})
\item[(2)] a tangle sum of two vertical short twist regions
\item[(3)] a tangle sum of a horizontal long twist region and a vertical short twist region 
\end{itemize}
\label{spectangdef}


\noindent To look for such tangles in $D(K) \subseteq S^{2}$, we look for simple closed curves in the plane that intersect $D(K)$ exactly four times and that contain a special tangle on one side of the curve. Equivalently, the special tangles of $D(K)$ can be found in the all-A state $H_{A}$ by looking for all-A circles that are incident to A-segments from a pair of twist regions from the tangle sums mentioned above. We call these all-A circles \emph{special circles (SCs)} of $H_{A}$. See Fig.~\ref{specialtangles} for depictions of special tangles and special circles. Let $st(D)$ denote the number of special tangles in $D(K)$ (or, equivalently, the number of special circles in $H_{A}$). 
\end{definition}

\begin{remark}
The advantage to looking for special circles in $H_{A}$, as opposed to looking for special tangles in $D(K)$, is that special circles are necessarily disjoint. Special tangles, on the other hand, can share one or both twist regions with another special tangle. 
\end{remark}

\begin{figure}
	\centering
		\def\svgwidth{4in}
		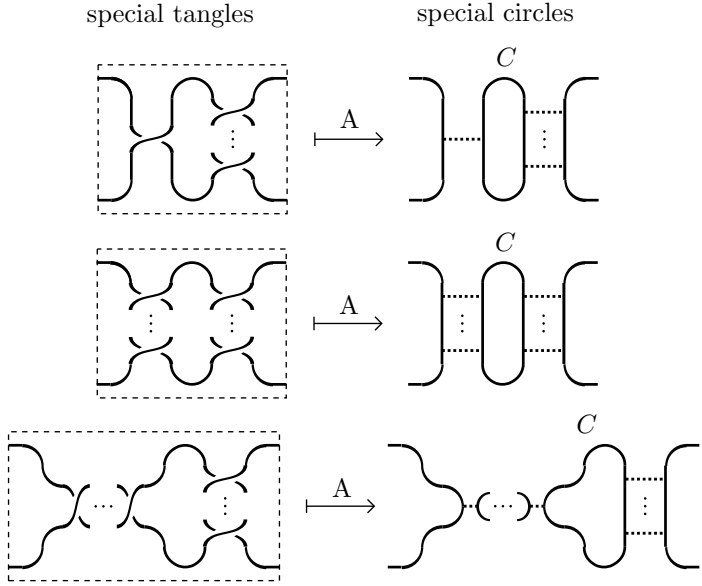
	\caption{Special tangles of $D(K)$ and the corresponding special circles, $C$, of $H_{A}$.}
	\label{specialtangles}
\end{figure}

\noindent By combining results from \cite{New}, \cite{Guts}, and \cite{Lackenby}, we get the following key result:

\begin{theorem}[Corollary 1.4 of \cite{New}, Theorem from Appendix of \cite{Lackenby}]
Let $D(K)$ be a connected, prime, A-adequate link diagram that satisfies the TELC and contains $t(D) \geq 2$ twist regions. Then $K$ is hyperbolic and:
\label{Cor}
\begin{equation}
-v_{8}\cdot\chi(\mathbb{G}_{A}') \leq \mathrm{vol}(S^{3}\backslash K) < 10v_{3}\cdot(t(D)-1).
\end{equation}
\end{theorem}


\section{Volume Bounds for A-Adequate Links}

\subsection{Twist Regions, State Circles, and $\mathbb{G}_{A}'$}
\label{twisty}

We begin with a study of the twist regions of an A-adequate link diagram $D(K)$ that satisfies the TELC. Because long and short resolutions are not distinguishable when there is only one crossing in a twist region, we will begin by considering the case of one-crossing twist regions. See Fig.~\ref{resolutions} for a one-crossing twist region and its A-resolution. Let $C_{1}$ and $C_{2}$ denote the (portions of the) relevant all-A circles in $H_{A}$. Since $D(K)$ is A-adequate, then $C_{1} \neq C_{2}$. Since $D(K)$ satisfies the TELC, then there can be no other additional A-segments between $C_{1}$ and $C_{2}$. Thus, the edge of $\mathbb{G}_{A}$ corresponding to this one-crossing twist region can never be a redundant parallel edge and, therefore, will always appear in $\mathbb{G}_{A}'$. 

\begin{remark} Let $t_{1}(D)$ denote the number of one-crossing twist regions in $D(K)$. By what was said in the above paragraph, $t_{1}(D)$ is also the number of edges in $\mathbb{G}_{A}'$ that come from the one-crossing twist regions of $D(K)$. 
\label{t1}
\end{remark}

Let us now consider twist regions that have at least two crossings (the short and long twist regions). See the right side of Fig.~\ref{longshort} for a twist region and its short resolution. If we again use $C_{1}$ and $C_{2}$ to denote the (portions of the) the relevant all-A circles, then the A-adequacy of $D(K)$ implies that $C_{1}\neq C_{2}$ and the TELC implies that there can be no other A-segments between $C_{1}$ and $C_{2}$ (besides those of the short resolution). Furthermore, note that a short twist region will always create redundant parallel edges in $\mathbb{G}_{A}$ since the parallel A-segments of $H_{A}$ join the same pair of state circles. Thus, all but one of these edges is removed when forming $\mathbb{G}_{A}'$. Said another way, there will be one edge of $\mathbb{G}_{A}'$ per short twist region of $D(K)$. 

\begin{remark} Let $t_{s}(D)$ denote the number of short twist regions in $D(K)$. By what was said in the above paragraph, $t_{s}(D)$ is also the number of edges in $\mathbb{G}_{A}'$ that come from the short twist regions in $D(K)$.
\label{tshort}
\end{remark}

See the left side of Fig.~\ref{longshort} for a twist region and its long resolution. We will use $C_{1}$ and $C_{2}$ to denote the upper and lower (portions of the) relevant state circles. If there are three or more crossings in the twist region being considered, then it must necessarily be the case that none of the corresponding edges in $\mathbb{G}_{A}$ are lost in the reduction to form $\mathbb{G}_{A}'$. If there are two crossings in the twist region, then the TELC implies that $C_{1}\neq C_{2}$ because, otherwise, we would have a two-edge loop in $\mathbb{G}_{A}$ coming from a long twist region. As a result, we have that no edges of $\mathbb{G}_{A}$ coming from long resolutions are removed when forming $\mathbb{G}_{A}'$. 

Recall that a long resolution will consist of (portions of) two state circles joined by a path of A-segments and (small) all-A circles.

\begin{definition} We call each (small) all-A circle in the interior of the long resolution a \emph{small inner circle (SIC)}. The remaining all-A circles in the rest of $H_{A}$ will simply be called \emph{other circles (OCs)}.  
\label{acircletypes}
\end{definition}

\noindent \textbf{Notation:} Let $t_{l}(D)$ denote the number of long twist regions in $D(K)$ and let $e_{l}(\mathbb{G}_{A}')$ denote the number of edges in $\mathbb{G}_{A}'$ coming from long twist regions. 

\bigskip

By inspection, it can be seen that the number of A-segments in the long resolution is always one greater than the number of small inner circles in the long resolution. Since this phenomenon occurs for each long resolution, then we have that:

\begin{equation}
\#\left\{SICs\right\}=e_{l}(\mathbb{G}_{A}')-t_{l}(D).
\label{ic}
\end{equation}

\subsection{Computation of $-\chi(\mathbb{G}_{A}')$}

\begin{lemma} Let $D(K)$ be a connected A-adequate link diagram that satisfies the TELC. Then we have that:
\label{chilemma}
\begin{equation}
-\chi(\mathbb{G}_{A}')=t(D)-\#\left\{OCs\right\}.
\end{equation}
\end{lemma}

\begin{proof}
By Remark~\ref{circleremark} and Definition~\ref{acircletypes}, we get:
\begin{eqnarray}
-\chi(\mathbb{G}_{A}') & = & e(\mathbb{G}_{A}')-v(\mathbb{G}_{A}')\nonumber\\
\ &  = & e(\mathbb{G}_{A}')-\#\left\{\text{all-A\ state\ circles}\right\}\nonumber\\
\ & =  & e(\mathbb{G}_{A}')-\#\left\{SICs\right\}-\#\left\{OCs\right\}.
\label{chistart}
\end{eqnarray}

\noindent Looking at how the twist regions of $D(K)$ were partitioned in Section~\ref{twisty}, we get:
\begin{equation}
t(D)=t_{1}(D)+t_{s}(D)+t_{l}(D).
\label{twistclasses}
\end{equation}

\noindent Next, Remark~\ref{t1} and Remark~\ref{tshort} imply that:
\begin{equation}
e(\mathbb{G}_{A}')=t_{1}(D)+t_{s}(D)+e_{l}(\mathbb{G}_{A}').
\label{edgeclasses}
\end{equation}

\noindent By substituting Eq.~(\ref{edgeclasses}) and Eq.~(\ref{ic}) into Eq.~(\ref{chistart}) and then using Eq.~(\ref{twistclasses}), we get:
\begin{eqnarray}\label{chi}
-\chi(\mathbb{G}_{A}')& = & e(\mathbb{G}_{A}')-\#\left\{SICs\right\}-\#\left\{OCs\right\}\nonumber\\
& = & [t_{1}(D)+t_{s}(D)+e_{l}(\mathbb{G}_{A}')]-[e_{l}(\mathbb{G}_{A}')-t_{l}(D)]-\#\left\{OCs\right\}\nonumber\\
& = & t(D)-\#\left\{OCs\right\}.
\end{eqnarray}\\
\end{proof}

\subsection{Special Circles and Special Tangles}


\begin{lemma}\label{tanglemma}
Let $D(K)$ be a connected, prime, A-adequate link diagram that satisfies the TELC and contains $t(D)\geq2$ twist regions. Furthermore, assume that $D(K)$ is not the link diagram depicted in Fig.~\ref{twistknot}. Then: 
\begin{itemize}
\item[(1)] each $OC$ of $H_{A}$ is incident to A-segments from at least two twist regions of $D(K)$ 
\item[(2)] the $OCs$ of $H_{A}$ that are incident to A-segments from exactly two twist regions of $D(K)$ are precisely the special circles ($SCs$) of $H_{A}$.
\end{itemize} 
\end{lemma}

\begin{remark}
Given the lemma above, notice that $st(D)=0$ (there are no $SCs$) for link diagrams $D(K)$ whose $OCs$ are all incident to A-segments from at least three twist regions of $D(K)$. 
\end{remark}

\begin{figure}
	\centering
		\includegraphics[width=1in]{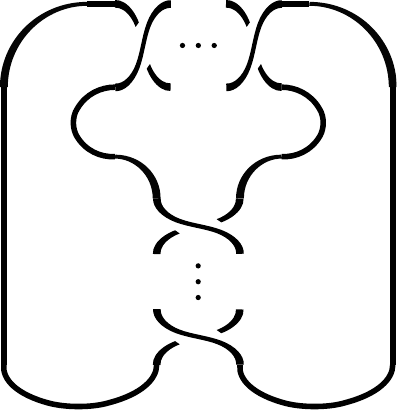}
	\caption{The exceptional link diagram consisting of two long twist regions, each of which must contain at least three crossings.}
	\label{twistknot}
\end{figure}


\begin{proof}
See Fig.~\ref{othertwist} for schematic depictions of the A-resolutions of the twist regions of $D(K)$. Let $C$ denote an $OC$ of the all-A state $H_{A}$. Such a circle must exist because, otherwise, we would have that $H_{A}$ is a cycle of small inner circles and A-segments. Since this all-A state corresponds to the standard $(2,p)$-torus link diagram, then we would get a contradiction of the assumption that $t(D)\geq2$.   

\begin{figure}
	\centering
	\def\svgwidth{3in}
		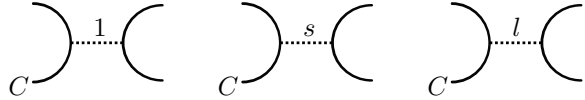
	\caption{A schematic depiction of the A-resolutions of the three possible types of twist regions incident to an $OC$, call it $C$, of $H_{A}$: one-crossing (left), short (middle), and long (right). The labels 1, $s$, and $l$ are used not only to indicate the type of twist region resolution, but also to distinguish these schematic segments from A-segments.}
	\label{othertwist}
\end{figure}

Suppose $C$ is an $OC$ incident to no A-segments. Then $C$ corresponds to a standard unknotted component of $D(K)$. Thus, either $D(K)$ is not connected or $D(K)$ is the standard unknot diagram. In either case we get a contradiction, given the assumptions that $D(K)$ is connected and contains $t(D)\geq2$ twist regions. 

Next, suppose $C$ is an $OC$ incident to A-segments from a single twist region. First, suppose that the A-segments of this twist region both start and end at $C$. If this twist region were a one-crossing twist region or a short twist region, then (recalling Fig.~\ref{resolutions} and Fig.~\ref{longshort} if needed) we get a contradiction of the assumption that $D(K)$ is A-adequate. If this twist region were a long twist region, then we get a contradiction of the assumption that $D(K)$ is connected and contains $t(D)\geq2$ twist regions. Second, suppose that the A-segments of this twist region start at $C$ and end at another all-A circle, call it $C'$. Then consider the portion of $H_{A}$ corresponding to Fig.~\ref{othertwist} and recall that $C$ and $C'$ are closed curves. If $C'$ were incident to no other additional twist region resolutions, then we get a contradiction of the assumption that $D(K)$ is connected and contains $t(D)\geq2$ twist regions. If $C'$ were incident to additional twist region resolutions, then we get a contradiction of the assumption that $D(K)$ is prime. This proves the first assertion of the lemma.

Now suppose that $C$ is an $OC$ incident to A-segments from exactly two twist regions. If the A-segments of a twist region both start and end at $C$, then this twist region must be a long twist region because (as seen in the paragraph above) we would otherwise get a contradiction of the assumption that $D(K)$ is A-adequate. 

Consider the case where both (long) twist regions give A-segments that start and end at $C$. The first three possibilities are depicted in Fig.~\ref{twoself}. As the rectangular dashed closed curves in the figure indicate, we get a contradiction of the primeness of the corresponding diagram $D(K)$. The fourth and final possibility is depicted in Fig.~\ref{twoselfok}. Translating back to $D(K)$, we get the exceptional link diagram depicted in Fig.~\ref{twistknot}. Note that, by the TELC, it must be the case that there are at least three crossings per (long) twist region. This is because, otherwise, we would have a two-edge loop whose edges do not correspond to crossings of a short twist region. Recall that the link diagram of Fig.~\ref{twistknot} has been excluded from consideration. 

\begin{figure}
	\centering
	\def\svgwidth{3.5in}
		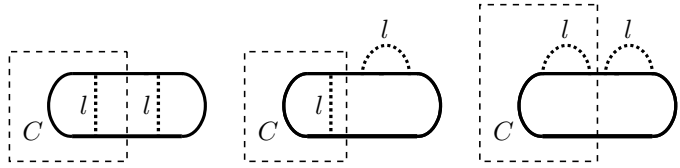
	\caption{Three possibilities for an $OC$, call it $C$, with two incident long resolutions that start and end at $C$.}
	\label{twoself}
\end{figure}

\begin{figure}
	\centering
	\def\svgwidth{1in}
		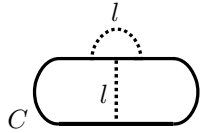
	\caption{The fourth possibility for an $OC$, call it $C$, with two incident long resolutions that start and end at $C$.}
	\label{twoselfok}
\end{figure}

Next, consider the case where one (long) twist region gives A-segments that start and end at $C=C_{1}$ and the other twist region gives A-segments that start at $C_{1}$ and end at another all-A circle, call it $C_{2}$. The three possibilities are depicted in Fig.~\ref{twomix}. As the rectangular dashed closed curves indicate, we get a contradiction of the primeness of the corresponding diagram $D(K)$.  

\begin{figure}
	\centering
	\def\svgwidth{4.5in}
		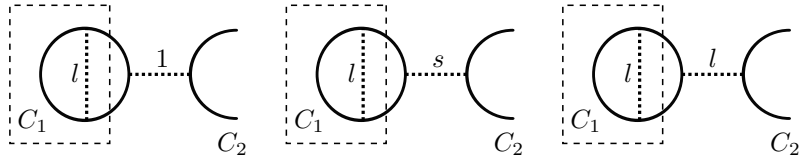
	\caption{Three possibilities for an $OC$, call it $C_{1}$, with two incident twist region resolutions, one resolution from a long twist region that starts and ends at $C_{1}$ and the other resolution connecting to a different state circle $C_{2}$.}
	\label{twomix}
\end{figure}

\begin{figure}
	\centering
	\def\svgwidth{1.5in}
		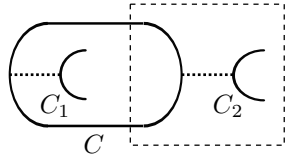
	\caption{An $OC$, call it $C$, with two incident twist region resolutions, one inside $C$ and one outside $C$.}
	\label{twocomp}
\end{figure}

\begin{figure}
	\centering
	\def\svgwidth{5in}
		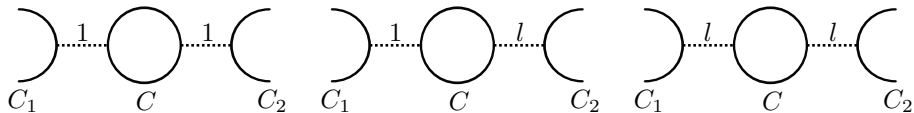
	\caption{Three possibilities for an ``$OC$'', call it $C$, with two incident twist region resolutions, one connecting $C$ to $C_{1}$ and the other connecting $C$ to $C_{2}$.}
	\label{twoout}
\end{figure}

Finally, consider the case where both twist regions give A-segments that start at $C$ and end at all-A circles, call them $C_{1}$ and $C_{2}$, that are different from $C$ (but where $C_{1}=C_{2}$ is possible in some cases). The first possibility, that $C_{1}$ and $C_{2}$ are on opposite sides of $C$, is depicted in Fig.~\ref{twocomp}. As the rectangular dashed closed curve indicates, we get a contradiction of the primeness of the corresponding diagram $D(K)$. Thus, $C_{1}$ and $C_{2}$ must be on the same side of $C$. The choice of side is irrelevant, however, since $D(K) \subseteq S^2$. The first three possibilities are depicted in Fig.~\ref{twoout}. None of these cases are possible because, by translating back to $D(K)$ (and recalling Fig.~\ref{resolutions} and Fig.~\ref{longshort} if needed), the A-segments from the two twist regions actually come from the same long twist region. This makes $C$ a small inner circle rather than an $OC$, a contradiction. The three remaining possibilities are depicted in Fig.~\ref{twooutmore}.   

\begin{remark} Note that, by the TELC, it must be the case that $C_{1} \neq C_{2}$ in the left and middle diagrams of Fig.~\ref{twooutmore}. However, because a long resolution involves a path of at least two A-segments, then it is possible that $C_{1}=C_{2}$ in the right diagram of Fig.~\ref{twooutmore}. It is also important to note that, in all three diagrams, the remaining twist region resolutions and all-A circles (not depicted) must somehow join $C_{1}$ to $C_{2}$ in a second way. This is because, otherwise, there would exist a simple closed curve that cuts $C$ in half and separates $C_{1}$ from $C_{2}$, a contradiction of the primeness of the corresponding diagram $D(K)$. \label{specialremark}  
\end{remark}

Assuming the conditions laid out in the above remark are satisfied, notice that the three possibilities in Fig.~\ref{twooutmore} do not give a contradiction of the assumptions of Lemma~\ref{tanglemma}. Equally as important, notice that these three possibilities correspond exactly to the three types of special circles depicted on the right side of Fig.~\ref{specialtangles} (and to the three types of special tangles depicted on the left side of Fig.~\ref{specialtangles}).
\end{proof}

\begin{figure}
	\centering
	\def\svgwidth{5in}
		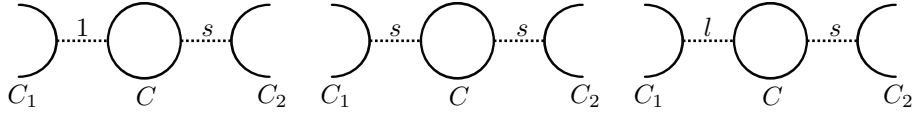
	\caption{Three remaining possibilities for an $OC$, call it $C$, with two incident twist region resolutions, one connecting $C$ to $C_{1}$ and the other (a short twist region resolution) connecting $C$ to $C_{2}$.}
	\label{twooutmore}
\end{figure}

\subsection{Volume Bounds in Terms of $t(D)$ and $st(D)$}


In this section, we will shift perspective from the link diagram $D(K)$ and its all-A state $H_{A}$ to the reduced all-A graph $\mathbb{G}_{A}'$. By combining some graph theory with our previous computation of $-\chi(\mathbb{G}_{A}')$ (Lemma~\ref{chilemma}) and our newly acquired knowledge about special circles (Lemma~\ref{tanglemma}), we will prove the Main Theorem (Theorem~\ref{mainthm}). 

\begin{definition} Let $G$ be a graph. We call $G$ \emph{simple} if it contains neither one-edge loops connecting a vertex to itself nor multiple edges connecting the same pair of vertices. 
\end{definition}

\noindent \textbf{Notation:} For $G$ a simple graph, let $V(G)$ denote its vertex set and let $E(G)$ denote its edge set. Furthermore, let deg$(v)$ denote the degree of the vertex $v$, that is, the number of edges incident to $v$.

\begin{proof}[Proof of the Main Theorem]
We will begin by using Theorem 2.1 of \cite{Graph} which states that, for $G$ a simple graph:
\begin{equation}\label{FTGT}
\sum_{v\in V(G)}\mathrm{deg}(v)=2\left|E(G)\right|.
\end{equation}

\noindent Our strategy will be to apply this result to the reduced all-A graph $\mathbb{G}_{A}'$. We can do this because A-adequacy of $D(K)$ implies that $\mathbb{G}_{A}'$ will not contain any loops and the fact that $\mathbb{G}_{A}'$ is reduced implies that $\mathbb{G}_{A}'$ will not contain any multiple edges. 

By Remark~\ref{circleremark}, Definition~\ref{acircletypes}, and Lemma~\ref{tanglemma}, we may partition $V(\mathbb{G}_{A}')$ into three types of vertices:


\begin{itemize}
\item[(1)] those corresponding to small inner circles ($SICs$), 
\item[(2)] those corresponding to special circles ($SCs$), which are $OCs$ that are incident to A-segments from exactly two twist regions, and
\item[(3)] those corresponding to the $OCs$ that are incident to A-segments from three or more twist regions (\emph{remaining OCs}).
\end{itemize}

Recall that, as said in the paragraph following Remark~\ref{tshort}, all edges corresponding to a long resolution survive the reduction to $\mathbb{G}_{A}'$. Thus, we have that deg$(v)=2$ for $v$ corresponding to a small inner circle. (See the left side of Fig.~\ref{longshort}.) Also recall that, as said in the paragraph preceding Remark~\ref{t1}, the edge corresponding to a one-crossing twist region survives the reduction to $\mathbb{G}_{A}'$. Finally, as said in the paragraph following Remark~\ref{t1}, only a single edge coming from a short twist region survives the reduction to $\mathbb{G}_{A}'$. By applying this knowledge to Fig.~\ref{twooutmore}, we see that deg$(v)=2$ for $v$ corresponding to a special circle. Similarly, we can see that deg$(v)\geq 3$ for $v$ corresponding to a \emph{remaining OC}. By translating Eq.~(\ref{FTGT}) to our setting, we get:
\begin{eqnarray}
2\cdot e(\mathbb{G}_{A}') & = & \sum_{SICs}\mathrm{deg}(v)+\sum_{SCs} \mathrm{deg}(v)+\sum_{remaining\ OCs}\mathrm{deg}(v)\nonumber \\
\ & = & 2\cdot\left[\#\left\{SICs\right\}\right]+2\cdot st(D)+\sum_{remaining\ OCs}\mathrm{deg}(v).
\label{sum}
\end{eqnarray}
\noindent Substituting Eq.~(\ref{edgeclasses}) and Eq.~(\ref{ic}) into Eq.~(\ref{sum}), we get:
\begin{eqnarray}
2\cdot t_{1}(D)+2\cdot t_{s}(D)+2\cdot e_{l}(\mathbb{G}_{A}') & = & 2\cdot\left[e_{l}(\mathbb{G}_{A}')-t_{l}(D)\right]\nonumber\\ 
\ & \ & + 2\cdot st(D)+ \sum_{remaining\ OCs}\mathrm{deg}(v).
\end{eqnarray}
\noindent By canceling, rearranging terms, and using Eq.~(\ref{twistclasses}), we end up with the following:
\begin{eqnarray}
2\cdot st(D)+\sum_{remaining\ OCs}\mathrm{deg}(v) & = & 2\cdot t_{1}(D)+2\cdot t_{s}(D)+2\cdot t_{l}(D)\nonumber\\
\ & = & 2t(D).
\label{degree}
\end{eqnarray}
\noindent Recall that deg$(v)\geq 3$ for $v$ corresponding to a \emph{remaining OC}. Thus, we get:
\begin{equation}
2\cdot st(D)+3\cdot\left[\#\left\{remaining\ OCs\right\}\right] \leq 2t(D).
\end{equation}
\noindent Adding $st(D)$, the number of $OCs$ that are not \emph{remaining OCs}, to both sides allows us to write the above inequality in terms of the total number of $OCs$ as:
\begin{eqnarray}
3\cdot\left[\#\left\{OCs\right\}\right] & = & 3\cdot st(D)+3\cdot\left[\#\left\{remaining\ OCs\right\}\right]\nonumber\\
\ & \leq & 2t(D)+st(D).
\end{eqnarray}
\noindent Combining this inequality with Lemma~\ref{chilemma} gives:
\begin{eqnarray}
-\chi(\mathbb{G}_{A}') & = & t(D)-\#\left\{OCs\right\}\nonumber\\
\ & \geq & t(D)-\left[\frac{2}{3}\cdot t(D) + \frac{1}{3}\cdot st(D)\right]\nonumber\\
\ & = & \frac{1}{3} \cdot \left[t(D)-st(D)\right].
\label{char}
\end{eqnarray}
\noindent Finally, by applying Inequality~(\ref{char}) to Theorem~\ref{Cor}, we get the desired volume bounds. 

Furthermore, notice that Eq.~(\ref{degree}) implies that $t(D) \geq st(D)$. Thus, we have that the lower bound on volume is always nonnegative and is positive precisely when there exists at least one \emph{remaining OC}. Looking at Eq.~(\ref{degree}) from another perspective, note that if $t(D)=st(D)$, then there are can be no \emph{remaining OCs} in the all-A state $H_{A}$. Hence, the only types of $OCs$ in this case are special circles. Since each special circle is incident to exactly two twist region resolutions (and since the conditions mentioned in Remark~\ref{specialremark} must be satisfied), then the all-A state $H_{A}$ must form a cycle alternating between special circles and twist region resolutions. But recall that special tangles (which correspond to special circles) are alternating tangles. Hence, by cyclically fusing these tangles together, we form an alternating link diagram. Consequently, in the case that $t(D)=st(D)$ (which forces the lower bound of the Main Theorem to be zero), Theorem $2.2$ of \cite{AgolStorm} can be used to provide a lower bound of $\displaystyle \frac{v_{8}}{2}\cdot(t(D)-2)$ on volume. \end{proof}

\begin{corollary}\label{mcor}
Let $D(K)$ satisfy the hypotheses of Theorem~\ref{mainthm}. Furthermore, assume that each $OC$ of $H_{A}$ has at least $m\geq3$ incident twist region resolutions. Then $K$ is hyperbolic, $st(D)=0$, and the complement of $K$ satisfies the following volume bounds:
\begin{equation}
\dfrac{m-2}{m}\cdot v_{8}\cdot t(D) \leq \mathrm{vol}(S^3\backslash K) < 10v_{3}\cdot(t(D)-1).
\end{equation}
\end{corollary}
 
\begin{remark} Notice that, as $m \rightarrow \infty$, the lower bound in the corollary above approaches $v_{8} \cdot t(D)$. Hence, the coefficients of $t(D)$ in the upper and lower bounds differ by a multiplicative factor of 2.7701$\ldots$ (in the limit). 
\end{remark}

\begin{proof} We will prove this result by modifying what needs to be modified in the above proof of the Main Theorem. First, the assumption that each $OC$ has at least $m\geq3$ incident twist region resolutions implies, by Lemma~\ref{tanglemma}, that special circles cannot exist, so we have $st(D)=0$. This assumption also implies that deg$(v)\geq m\geq3$ for $v$ corresponding to an $OC$ (which must be a \emph{remaining OC}). By incorporating these conditions into Eq.~(\ref{degree}), we get:
\begin{eqnarray}
m\cdot\#\left\{OCs\right\} & \leq & 2 \cdot st(D)+\sum_{remaining\ OCs}\mathrm{deg}(v)\nonumber\\
\ &= & 2t(D).
\end{eqnarray}
\noindent Combining this inequality with Lemma~\ref{chilemma} gives:
\begin{eqnarray}
-\chi(\mathbb{G}_{A}') & = & t(D)-\#\left\{OCs\right\}\nonumber\\
\ & \geq & t(D)-\frac{2}{m}\cdot t(D)\nonumber\\
\ & = & \frac{m-2}{m} \cdot t(D).
\label{chicorbound}
\end{eqnarray}
\noindent Finally, by applying the above inequality to Theorem~\ref{Cor}, we get the desired volume bounds. \end{proof}

\begin{figure}
	\centering
		\def\svgwidth{3.5in}
		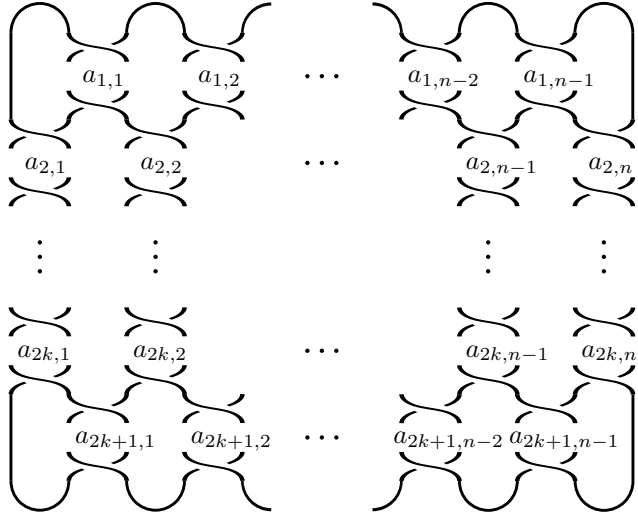
	\caption{A schematic depiction of a special $2n$-plat diagram with $m=2k+1$ rows of twist regions, where the entry in the $i^{\mathrm{th}}$ row and $j^{\mathrm{th}}$ column is a twist region containing $a_{i,j}$ crossings (counted with sign). The twist regions depicted above are negative twist regions. Having $a_{i,j}>0$ instead will reflect the crossings in the relevant twist region.}
	\label{newplat}
\end{figure}

\section{Volume Bounds for A-Adequate Plats}

To provide collections of links that satisfy the hypotheses of the Main Theorem (Theorem~\ref{mainthm}) and to seek to improve the lower bounds on volume, we will now investigate certain families of A-adequate plat diagrams. 

\subsection{Background on Plat Closures}
\label{secplat}
 
\begin{definition}Given a braid $\beta$ in the even-stringed braid group $B_{2n}$, we can form the \emph{plat closure} of $\beta$ by connecting string position $2i-1$ with string position $2i$ for each $1 \leq i \leq n$ by using trivial semicircular arcs at the top and bottom of these string positions. See Fig.~\ref{newplat} for a schematic depiction of the type of plat closure, call it a \emph{special plat closure}, that we will consider in this paper. 
\end{definition}

\noindent \textbf{Notation:} Let the special plat closure of $\beta \in B_{2n}$ have $m=2k+1$ rows of twist regions. Specifically, if we number the rows of twist regions from the top down, then there are $k+1$ odd-numbered rows, each of which contains $n-1$ twist regions, and $k$ even-numbered rows, each of which contains $n$ twist regions. Index the twist regions according to row and column (where by column we really mean the left-to-right ordering of twist regions in a given row). Denote the number of twist regions in row $i$ and column $j$ (counted with sign)  by $a_{i,j}$, where $1 \leq i \leq m$ and:
$$\left\{
        \begin{array}{ll}
   1 \leq j \leq n-1 & \mathrm{if} \ i\ \mathrm{is\ odd}\\
	 1 \leq j \leq n & \mathrm{if}\ i\ \mathrm{is\ even.}\\
        \end{array}
    \right.$$

\noindent Refer back to Fig.~\ref{newplat} to see this notation in use. 

\begin{remark} For the remainder of this paper, the term ``for all $i$ and $j$'' will be assumed to apply to $i$ and $j$ that satisfy the above conditions. 
\end{remark}

\begin{definition} 
Let $D(K)$ denote a special plat closure of a braid $\beta \in B_{2n}$, where $n \geq 3$, that contains $2k+1$ rows of twist regions, where $k \geq 1$. Then we call $D(K)$ a \emph{strongly negative plat diagram} if $a_{i,j} \leq -3$ in odd-numbered rows and $a_{i,j} \leq -2$ in even-numbered rows. Similarly, we call $D(K)$ a \emph{mixed-sign plat diagram} if $a_{i,j} \leq -3$ or $a_{i,j} \geq 1$ in odd-numbered rows and $a_{i,j} \leq -2$ in even-numbered rows. See the left side of Fig.~\ref{posplat} for an example of a strongly negative plat diagram and see the left side of Fig.~\ref{mixedplat} for an example of a mixed-sign plat diagram.  
\end{definition} 

\begin{figure}
	\centering
	\def\svgwidth{3.5in}
		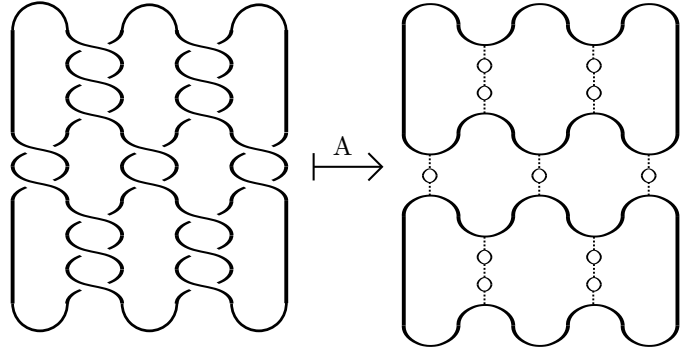
	\caption{An example of a strongly negative plat diagram and its all-A state.}
	\label{posplat}
\end{figure}

\begin{figure}
	\centering
	\def\svgwidth{3.5in}
		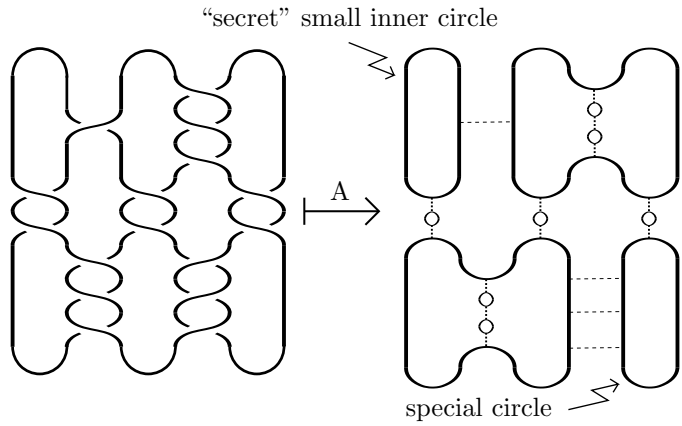
	\caption{An example of a mixed-sign plat diagram and its all-A state. Note that the diagram above is obtained from the strongly negative plat diagram of Fig.~\ref{posplat} by changing the first negative twist region with three crossings to a positive twist region with a single crossing and changing the last negative twist region with three crossings to a positive twist region with three crossings. These changes create a ``secret'' small inner circle and a special circle, respectively.}
	\label{mixedplat}
\end{figure}

\begin{remark}
\label{bridgeremark}
When $n=2$ we have that $D(K)$ represents a two-bridge link $K$. Using the fact that two-bridge links are alternating, let $D={D}_{alt}(K)$ denote a reduced alternating diagram of $K$. It will be shown later that the plats considered in this work are all hyperbolic. Therefore, by Theorem B.3 of \cite{Bridge}, we get the following volume bounds:
\begin{equation}
2v_{3}\cdot t(D)-2.7066 < \mathrm{vol}(S^3\backslash K) < 2v_{8}\cdot\left(t(D)-1\right).
\end{equation}
\noindent Note that the coefficients of $t(D)$ in the upper and lower bounds above differ by a multiplicative factor of $3.6100\ldots$.
Since we have the above (better) volume bounds when $n=2$, then will assume for the remainder of this paper (as we have done with the definitions of strongly negative and mixed-sign plat diagrams) that $n\geq3$. 
\end{remark}


\subsection{Volume Bounds for Strongly Negative Plats in Terms of $t(D)$}

\begin{theorem}\label{posthm}
Let $D(K)$ be a strongly negative plat diagram. Then $D(K)$ is a connected, prime, A-adequate diagram that satisfies the TELC, contains $t(D)\geq7$ twist regions, and contains $st(D)=0$ special tangles. Furthermore, $K$ is hyperbolic and the complement of $K$ satisfies the following volume bounds:
\begin{equation}
\frac{4v_{8}}{5}\cdot(t(D)-1)+\frac{v_{8}}{5} \leq \mathrm{vol}(S^3\backslash K) < 10v_{3}\cdot\left(t(D)-1\right).
\end{equation}
\end{theorem}

For an example of a plat diagram that satisfies the assumptions above, see the left side of Fig.~\ref{posplat}. Having such a figure in mind will help when considering the proof of this result. 

\begin{proof} Since $a_{i,j}\neq 0$ for all $i$ and $j$, then $D(K)$ must be a connected link diagram. See Fig.~\ref{newplat} for visual support. 

Since we have that $k \geq 1$, that $n \geq 2$, and that $a_{i,j} \neq 0$ for all $i$ and $j$, then by careful and methodical inspection we get that $D(K)$ is prime. To see this, let $C$ denote a simple closed curve in the plane that intersects $D(K)$ twice transversely and let $p$ be an arbitrary base point for $C$. Considering the possible locations of $p$ in $S^2 \backslash D(K)$ (perhaps using Fig.~\ref{posplat} to assist in visualization), it can be seen that it is impossible for $C$ to both close up and contain crossings on both sides. Thus, $D(K)$ is indeed a prime link diagram.  


By inspecting $H_{A}$, we get that $D(K)$ is A-adequate. To see this, first notice that the vertical A-segments between (the necessarily distinct) $OCs$ can never contribute to non-A-adequacy. Second, since $a_{i,j} \leq -2$ in odd-numbered rows, then the vertical A-segments within a given $OC$ either connect distinct small inner circles or connect an $OC$ to a small inner circle. Therefore, since no A-segment connects a circle to itself, then $D(K)$ is A-adequate. 

The assumptions that $a_{i,j} \leq -3$ in odd-numbered rows and $a_{i,j}\leq-2$ in even-numbered rows guarantee that $D(K)$ satisfies the TELC. To be specific, having $a_{i,j}\leq-2$ in even-numbered rows forces there to always be at least one small inner circle to act as a buffer between adjacent $OCs$, making it is impossible for two given $OCs$ to share any (let alone two) A-segments. Furthermore, notice that a small inner circle from an even-numbered row must always connect to a pair of distinct circles. Next, having $a_{i,j} \leq -3$ in odd-numbered rows guarantees that there are at least two inner circles for each odd-rowed twist region, which prevents an $OC$ from connecting to an interior small inner circle and then back to itself along another A-segment. Finally, by construction, it is impossible for a pair of small inner circles to share more than one A-segment. Since we have just shown that no two all-A circles share more than one A-segment, then the TELC is trivially satisfied. 

Since $n \geq 3$, $k \geq 1$, and $a_{i,j}\neq0$ for all $i$ and $j$, then $t(D)\geq7\geq2$. Combining this with what was shown above and using Theorem~\ref{Cor}, we can conclude that $K$ is hyperbolic. Inspection also shows that $st(D)=0$ because each $OC$ is incident to at least five twist region resolutions. 

It remains to show that $K$ satisfies the desired volume bounds. Since there is one $OC$ of $H_{A}$ corresponding to each odd row of twist regions in $D(K)$, then we have that $\#\left\{\mathrm{OCs}\right\}=k+1$. Applying Lemma~\ref{chilemma} gives: 
\begin{eqnarray}
-\chi(\mathbb{G}_{A}') & = & t(D)-\#\left\{OCs\right\}\nonumber\\
\ & = & t(D)-k-1.
\end{eqnarray}
We would now like to eliminate the dependence of $-\chi(\mathbb{G}_{A}')$ on $k$. Expand $t(D)$ as:
\begin{eqnarray}
\label{treg}
	t(D) &=& \#(\mathrm{odd}\text{-}\mathrm{numbered\ rows})\cdot\#(\mathrm{twist\ regions\ per\ odd\ row}) \nonumber\\
	\ &\ & +\#(\mathrm{even}\text{-}\mathrm{numbered\ rows})\cdot\#(\mathrm{twist\ regions\ per\ even\ row}) \nonumber\\
  \ &=& (k+1)(n-1)+kn.
\end{eqnarray}
\noindent Since $n \geq 3$, then $t(D)=(k+1)(n-1)+kn=2kn-k+n-1\geq5k+2$, which implies that $\displaystyle k \leq \frac{t(D)-2}{5}$. Thus, we get the following:
\begin{eqnarray}
\label{chipos}
-\chi(\mathbb{G}_{A}') & = & t(D)-k-1\nonumber\\
\ & \geq & t(D)-\left(\frac{t(D)-2}{5}\right)-1\nonumber\\
\ & = & \frac{4}{5}\cdot (t(D)-1)+\frac{1}{5}.
\end{eqnarray}
\noindent By applying Theorem~\ref{Cor}, we get the desired volume bounds. \end{proof}

\begin{remark}
It can be shown that, since $t(D)\geq7$, then the lower bound given by Theorem~\ref{posthm} is always sharper than the lower bound provided by applying Corollary~\ref{mcor} to strongly negative plats. 
\end{remark}

%
%
%
 %
%

\subsection{Volume Bounds for Mixed-Sign Plats in Terms of $t(D)$}
\label{mixitup}
 
Starting from a strongly negative plat diagram, we are able to form a mixed-sign plat diagram by iteratively replacing any of the negative twist regions in the odd-numbered rows with positive twist regions (which need only contain at least one crossing). For an example of this process, see how Fig.~\ref{posplat} turns into Fig.~\ref{mixedplat}. Notice that changing an arbitrary negative twist region of an odd-numbered row to a positive twist region will break the relevant $OC$ into two all-A circles. This is because a long twist region is changed to a one-crossing or short twist region. In the relevant part of the new all-A state, all but one of the new horizontal A-segments correspond to redundant parallel edges of $\mathbb{G}_{A}$. Thus, this entire new positive twist region corresponds to a single edge of $\mathbb{G}_{A}'$. These remarks hold true during every iteration of the procedure mentioned above.  

\bigskip

\noindent \textbf{Notation:} Let $t^{+}(D)$ and $t^{-}(D)$ denote the number of positive and negative twist regions in $D(K)$, respectively. 

\begin{theorem} 
\label{mixthm}
Let $D(K)$ be a mixed-sign plat diagram. Then $D(K)$ is a connected, prime, A-adequate link diagram that satisfies the TELC, contains $t(D)\geq3$ twist regions, and contains $st(D)\leq4$ special tangles. Furthermore, $K$ is hyperbolic and the complement of $K$ satisfies the following volume bounds:
\begin{equation}
\displaystyle \frac{v_{8}}{3}\cdot\left(2t^{-}(D)-1\right)-\frac{2v_{8}}{3} \leq \mathrm{vol}(S^3 \backslash K) < 10v_{3}\cdot\left(t(D)-1\right).
\end{equation}
\noindent If we also have that $D(K)$ contains at least as many negative twist regions as it does positive twist regions, then: 
\begin{equation}
\displaystyle \frac{v_{8}}{3}\cdot\left(t(D)-1\right)-\frac{2v_{8}}{3} \leq \mathrm{vol}(S^3 \backslash K) < 10v_{3}\cdot\left(t(D)-1\right).
\end{equation}
\end{theorem} 

\begin{proof} The proofs of the connectedness and primeness of $D(K)$ are the same as those found in the proof of Theorem~\ref{posthm} and the proof that $D(K)$ is A-adequate is very similar. The only new observation that is needed is that any horizontal A-segments coming from positive twist regions necessarily connect distinct all-A circles. The proof that $D(K)$ satisfies the TELC is also similar to that found in the proof of Theorem~\ref{posthm}, but two-edge loops may now exist. The new possibility that $a_{i,j}\geq1$ in odd-numbered rows will give rise to two-edge loops whenever $a_{i,j}\geq2$. These two-edge loops come from the same short twist region and are, therefore, allowed by the TELC. 

Inspection of $H_{A}$ shows that the mixed-sign plat diagrams contain at least $7-4=3$ twist regions. This is because having $a_{i, j}=1$ in any of the four corners of $D(K)$ means that the corresponding state circles of $H_{A}$ in those corners will be ``secret'' small inner circles rather than $OCs$ and, consequently, we may have to absorb at most four twist regions into existing negative (long) twist regions. See Fig.~\ref{mixedplat} for an example. 

Using what was shown above, we can apply Theorem~\ref{Cor} to conclude that $K$ is hyperbolic. Inspection of $H_{A}$ shows that special circles can actually occur in mixed-sign plat diagrams. However, special circles can only possibly occur at the four corners of the link diagram. This is because, by the assumption that $a_{i,j} \neq 0$ for all $i$ and $j$, all but at most the four corner $OCs$ must be incident to three or more twist region resolutions. See Fig.~\ref{mixedplat} for an example. Therefore, we have that $st(D)\leq4$. It remains to show that $K$ satisfies the desired volume bounds. 

Recall the observation that we may start with a strongly negative plat diagram and iteratively change any of the negative twist regions in the odd-numbered rows to positive twist regions. This creates either a new $OC$ or a new small inner circle. Thus, after changing any odd-rowed negative twist regions to positive twist regions, we have that $\#\left\{OCs\right\} \leq k+1+t^{+}(D)$. Applying Lemma~\ref{chilemma} gives: 
\begin{eqnarray}
-\chi(\mathbb{G}_{A}') & = & t(D)-\#\left\{OCs\right\}\nonumber\\
\ & \geq & t(D)-k-1-t^{+}(D)\nonumber\\
\ & = & t^{-}(D)-k-1.
\label{chik2}
\end{eqnarray}
\noindent Recall that, by construction, we can only have positive twist regions in odd-numbered rows. Thus, all of the even-numbered rows must still contain only negative twist regions. Said another way: 
\begin{eqnarray}
\label{treg2}
	t^{-}(D) & \geq & \#(\mathrm{even}\text{-}\mathrm{numbered\ rows})\cdot\#(\mathrm{twist\ regions\ per\ even\ row})\nonumber\\
	\ & = & k\cdot n.
\end{eqnarray}
\noindent Since $t^{-}(D) \geq kn$, then the assumption that $n \geq 3$ gives $k \leq \dfrac{t^{-}(D)}{n} \leq \dfrac{t^{-}(D)}{3}$. Therefore, we get:
\begin{eqnarray}
-\chi(\mathbb{G}_{A}') & \geq & t^{-}(D)-k-1\nonumber\\
\ & \geq & t^{-}(D)-\frac{t^{-}(D)}{3}-1\nonumber\\
\ & = & \frac{1}{3}\cdot (2t^{-}(D)-1)-\frac{2}{3}.
\end{eqnarray}
\noindent Now suppose that $D(K)$ contains at least as many negative twist regions as it does positive twist regions, so that we have $t^{-}(D) \geq t^{+}(D)$. This implies that:
\begin{eqnarray}
2t^{-}(D) & = & t^{-}(D)+t^{-}(D)\nonumber\\
\ & \geq & t^{-}(D) + t^{+}(D)\nonumber\\
\ & = & t(D),
\end{eqnarray} 
\noindent which then implies that:
\begin{eqnarray}\label{chimixed}
-\chi(\mathbb{G}_{A}') & \geq & \frac{1}{3}\cdot (2t^{-}(D)-1)-\frac{2}{3}\nonumber\\
\ & \geq & \frac{1}{3}\cdot (t(D)-1)-\frac{2}{3}.
\end{eqnarray}
\noindent By applying Theorem~\ref{Cor}, we have the desired volume bounds. \end{proof}

\begin{remark} Note that the lower bounds on volume in terms of $t^{-}(D)$ will be sharper than those in terms of $t(D)$ in the case that $D(K)$ is a mixed-sign plat with more negative twist regions than positive twist regions. Furthermore, as the disparity between the number of positive and negative twist regions increases, the lower bound on volume in terms of $t^{-}(D)$ will continue to improve over the bound in terms of $t(D)$. 
\end{remark}

To conclude our study of mixed-sign plats, we would like to find a sufficient condition to guarantee that such a plat contains at least as many negative twist regions as positive twist regions. 

\begin{proposition}
If a mixed-sign plat contains $m \geq 2n-1$ rows of twist regions, then this plat contains at least as many negative twist regions as positive twist regions. 
\end{proposition}

\begin{proof} Since the process to change a strongly negative plat into a mixed-sign plat may create a situation where seemingly different twist regions are actually part of a single twist region, then Eq.~(\ref{treg}) for strongly negative plats becomes the inequality $t(D)\leq(k+1)(n-1)+kn$ for mixed-sign plats. By Inequality~(\ref{treg2}), we also have that $t^{-}(D) \geq kn$. Combining this information, we get that: 
\begin{eqnarray}
t^{-}(D)+t^{+}(D) & = & t(D)\nonumber\\
\ & \leq & (k+1)(n-1)+kn\nonumber\\
\ & \leq & (k+1)(n-1)+t^{-}(D),
\end{eqnarray}
\noindent which implies that: 
\begin{eqnarray}
t^{+}(D) & \leq & (k+1)(n-1)\nonumber\\
\ & = & kn+n-k-1\nonumber\\
\ & \leq & t^{-}(D)+n-k-1.
\end{eqnarray}
\noindent Thus, to guarantee that $t^{-}(D) \geq t^{+}(D)$, we need that $n-k-1 \leq 0$. But this condition is equivalent to $k \geq n-1$ is equivalent to $m=2k+1 \geq 2n-1$. \end{proof}

\begin{remark}
It can be shown that, for mixed-sign plats that contain at least as many negative twist regions as positive twist regions, the lower bound found in Theorem~\ref{mixthm} is always slightly sharper than the lower bound provided by applying the Main Theorem (Theorem~\ref{mainthm}). \end{remark}

%
%
%
%
%

\section{Volume Bounds in Terms of the Colored Jones Polynomial} 

\begin{theorem}[\cite{HeadTail}, \cite{Stoimenow}]
Denote the $n^{th}$ colored Jones polynomial of a link $K$ by:
\begin{equation}
J_{K}^{n}(t)=\alpha_{n}t^{m_{n}}+\beta_{n}t^{m_{n}-1}+\cdots+\beta_{n}'t^{r_{n}+1}+\alpha_{n}'t^{r_{n}}.
\end{equation} 
Let $D(K)$ be a connected A-adequate link diagram. Then $\left|\beta_{n}'\right|$ is independent of $n$ for $n>1$. Specifically, for $n>1$, we have that: 
\begin{equation}\label{stable}
\left|\beta_{K}'\right|:=\left|\beta_{n}'\right|=1-\chi(\mathbb{G}_{A}').
\end{equation}
\end{theorem}

\begin{remark}
By combining the above result with Theorem~\ref{Cor}, we get that: 
\begin{equation}
v_{8} \cdot \left(\left|\beta_{K}'\right|-1\right)\leq \mathrm{vol}(S^{3}\backslash K)
\end{equation}
\noindent for the links considered in this paper. 
\end{remark}

Furthermore, by applying Eq.~(\ref{stable}) and Inequalities~(\ref{char}),~(\ref{chicorbound}),~(\ref{chipos}), and~(\ref{chimixed}), respectively, to Theorem~\ref{Cor}, we get the following respective results: 

\begin{proposition} Let $D(K)$ be a connected, prime, A-adequate link diagram that satisfies the TELC and contains $t(D)\geq2$ twist regions. Then $K$ is hyperbolic and:
\begin{equation}
\mathrm{vol}(S^{3}\backslash K) < 30v_{3}\cdot \left(\left|\beta_{K}'\right|-1\right) +10v_{3} \cdot \left(st(D)-1\right).
\end{equation}
\end{proposition}

\begin{proposition}
Let $D(K)$ satisfy the hypotheses of Theorem~\ref{mainthm}. Furthermore, assume that each $OC$ of $H_{A}$ has at least $m\geq3$ incident twist region resolutions. Then $K$ is hyperbolic and:
\begin{equation}
\mathrm{vol}(S^3\backslash K) < \dfrac{m}{m-2}\cdot 10v_{3}\cdot \left(\left|\beta_{K}'\right|-1\right)-10v_{3}.
\end{equation}
\end{proposition}

\begin{proposition} 
Let $D(K)$ be a strongly negative plat diagram. Then the link $K$ is hyperbolic and: 
\begin{equation}
\mathrm{vol}(S^3\backslash K)  <  \frac{25v_{3}}{2}\cdot (\left|\beta_{K}'\right|-1)-\frac{5v_{3}}{2}.
\end{equation}
\end{proposition}

\begin{proposition} 
Let $D(K)$ be a mixed-sign plat diagram that contains at least as many negative twist regions as it does positive twist regions. Then the link $K$ is hyperbolic and: 
\begin{equation}
\mathrm{vol}(S^3\backslash K) < 30v_{3}\cdot (\left|\beta_{K}'\right|-1)+20v_{3}.
\end{equation}
\end{proposition} 

\begin{remark}
The results in this section show that the links of the Main Theorem (including the strongly negative and mixed-sign plats) and the links of Corollary~\ref{mcor} satisfy a Coarse Volume Conjecture (\cite{Guts}, Section 10.4).  
\end{remark}

\section*{Acknowledgments}
I would like to thank Efstratia Kalfagianni for suggesting this project. I would also like to thank Faramarz Vafaee and Jessica Purcell for their helpful comments on earlier versions of this paper. Research was supported in part by RTG grant DMS-0739208 and NSF grant DMS-1105843.

\bibliography{mybib}
\bibliographystyle{plain}

\end{document}